\documentclass[11pt]{article}
\usepackage{amsfonts,amssymb,amsmath,amsthm}

\usepackage{bm}
\usepackage{palatino}
\usepackage{mathpazo}
\usepackage{inconsolata}
\usepackage[hidelinks]{hyperref}

\usepackage{xcolor}
\usepackage{varwidth}
\usepackage{multicol}

\begingroup
    \makeatletter
    \@for\theoremstyle:=definition,remark,plain\do{%
        \expandafter\g@addto@macro\csname th@\theoremstyle\endcsname{%
            \addtolength\thm@preskip\parskip}%
        }
\endgroup

\usepackage[margin=1in]{geometry}

\theoremstyle{definition}
\newtheorem{theorem}{Theorem}
\newtheorem{proposition}[theorem]{Proposition}
\newtheorem{lemma}[theorem]{Lemma}
\newtheorem{corollary}[theorem]{Corollary}

\numberwithin{theorem}{section}
\numberwithin{question}{section}
\numberwithin{definition}{section}
\numberwithin{example}{section}
\numberwithin{remark}{section}

\newcommand*\patchAmsMathEnvironmentForLineno[1]{%
  \expandafter\let\csname old#1\expandafter\endcsname\csname #1\endcsname
  \expandafter\let\csname oldend#1\expandafter\endcsname\csname end#1\endcsname
  \renewenvironment{#1}%
     {\linenomath\csname old#1\endcsname}%
     {\csname oldend#1\endcsname\endlinenomath}}%
\newcommand*\patchBothAmsMathEnvironmentsForLineno[1]{%
  \patchAmsMathEnvironmentForLineno{#1}%
  \patchAmsMathEnvironmentForLineno{#1*}}%
\AtBeginDocument{%
\patchBothAmsMathEnvironmentsForLineno{equation}%
\patchBothAmsMathEnvironmentsForLineno{align}%
\patchBothAmsMathEnvironmentsForLineno{flalign}%
\patchBothAmsMathEnvironmentsForLineno{alignat}%
\patchBothAmsMathEnvironmentsForLineno{gather}%
\patchBothAmsMathEnvironmentsForLineno{multline}%
}

\usepackage{parskip}
\usepackage{lineno}
\usepackage[small]{titlesec}

\usepackage[square,comma,numbers,sort&compress]{natbib}

\usepackage{color}
\definecolor{todocolor}{RGB}{205,235,139}
\definecolor{todo-idea}{RGB}{120,180,255}
\definecolor{todo-error}{RGB}{208,31,60}
\definecolor{todo-question}{RGB}{255,255,136}

\usepackage[colorinlistoftodos, color=todocolor, textsize=small]{todonotes}
\usepackage{colortbl}
\usepackage{tabularray}

\usepackage{ stmaryrd }
\usepackage{tikz}
\usepackage{pgfplots}
\pgfplotsset{width=10cm,compat=1.9}

\newcommand{\Av}{\mathcal{R}}

\renewenvironment{abstract}{
	\begin{list}{}%
	{\setlength{\rightmargin}{1in}%
	\setlength{\leftmargin}{1in}}%
	\item[]\ignorespaces\begin{small}}%
	{\end{small}\unskip\end{list}%
}

\newpagestyle{main}[\small]{
	\headrule{}
	\sethead[\usepage][][]
	{\sc }{}{\usepage}
}

\title{The Insertion Encoding of Restricted Growth Functions}

\usepackage{makecell}
\author{
	\begin{tabular}{c}
		\makecell{
			Christian Bean\\
			\small School of Computer Science\\
			\small and Mathematics\\
			\small Keele University\\
			\small Keele, United Kingdom\\
			\small \texttt{c.n.bean@keele.ac.uk}
		}
		\quad\quad
		\makecell{
			Paul C. Bell\\
			\small School of Computer Science\\
			\small and Mathematics\\
			\small Liverpool John Moores University\\
			\small Liverpool, United Kingdom\\
			\small \texttt{p.c.bell@ljmu.ac.uk}
		}\\ \\
		\makecell{
			Abigail Ollson\\
			\small School of Computer Science\\
			\small and Mathematics\\
			\small Keele University\\
			\small Keele, United Kingdom\\
			\small \texttt{a.n.ollson@keele.ac.uk}
		}
	\end{tabular}
}

\date{}

\usepackage{arydshln}
\usepackage{makecell}
\usepackage{mathtools} 
\usepackage{graphicx} 
\usepackage{subcaption} 
\usepackage{amsmath} 

\begin{document}
\maketitle

\begin{abstract}
	We adapt the vertical and horizontal insertion encodings of Cayley permutations to enumerate restricted growth functions, which are in bijection with unordered set partitions. For both insertion encodings, we fully classify the classes for which these languages are regular. For the horizontal insertion encoding, we also prove that the conditions to be regular are the same for restricted growth functions of matchings.
\end{abstract}

\section{Introduction}
\label{sec:intro}

A \emph{set partition} of size $n$ is a partition of $[n] = \{ 1, 2, \dots, n\}$ into disjoint non-empty subsets called \emph{blocks} whose union is $[n]$. For example, for the set $[4] = \{1, 2, 3, 4 \}$ a set partition of size $4$ into 2 blocks is $\{1, 3 \}, \{2, 4\}$, or an example into 3 blocks is $\{3\}, \{2,4\}, \{1\}$.

There are two types of set partitions, \emph{ordered} and \emph{unordered}. In ordered set partitions, the order of the blocks matters so $\{3\}, \{2,4\}, \{1\}$ is different from $\{3\}, \{1\}, \{2,4\}$ or $\{2,4\}, \{1\}, \{3\}$. In unordered set partitions, the order of the blocks does not matter so $\{3\}, \{2,4\}, \{1\}$ is the same as $\{3\}, \{1\}, \{2,4\}$ or $\{2,4\}, \{1\}, \{3\}$. For unordered set partitions, we will write the blocks in increasing order of their smallest element, so $\{3\}, \{2,4\}, \{1\}$ is written as $\{1\}, \{2,4\}, \{3\}$.

Ordered set partitions are in bijection with \emph{Cayley permutations} which are words $\pi \in \mathbb{N}^*$ such that every value between 1 and the maximum value of $\pi$ appears at least once. The value $i$ in the $j^{th}$ block of an ordered set partition corresponds to the $i^{th}$ index of the Cayley permutation having value $j$. For example, the ordered set partition $\{1, 3\}, \{2, 4\}$ corresponds to the Cayley permutation $1 2 1 2$, the ordered set partition $\{3\}, \{2, 4\}, \{1\}$ corresponds to $3 2 1 2$ and $\{1\}, \{2,4\}, \{3\}$ corresponds to $1 2 3 2$.

The Cayley permutations which are in bijection with unordered set partitions have the extra constraint that the first occurrence of a value $n$ must occur after an occurrence of every value smaller than $n$. These are called \emph{restricted growth functions} (RGFs). For example, $1 2 3 2$ and $1 1 2 1 3 4$ are RGFs but $3 2 1 2$ and $1 2 4 3 2$ are not because the first occurrence of 3 in $3 2 1 2$ occurs before the first occurrence of 1 and 2 and the first occurrence of 4 in $1 2 4 3 2$ occurs before the first occurrence of 3. We define the \emph{size} of an RGF to be the number of letters in the sequence and the \emph{height} of an RGF to be the maximum value in the sequence so the size of $1 1 2 1 3 4$ is 6 and the height is 4.

The \emph{standardisation} of a word $\pi$ is the Cayley permutation obtained by replacing all occurrences of the smallest value in $\pi$ with 1, the next smallest value with 2, and so on. For example, the standardisation of the word $677649$ is the Cayley permutation $233214$.
We say that a Cayley permutation $\pi$ \emph{contains} an occurrence of a Cayley permutation $\sigma$ if there exists a subsequence of $\pi$ which standardises to $\sigma$. For example, $1 2 1 3 2 1 4$ contains the Cayley permutation $2213$ because the subsequence $2214$ standardises to $2213$. Otherwise, if there are no subsequences of $\pi$ which standardise to $\sigma$, we say that $\pi$ \emph{avoids} $\sigma$. For example, the Cayley permutation $1213214$ does not contain any occurrences of the Cayley permutation $4321$. In this context, we often refer to the Cayley permutations being contained or avoided as \emph{patterns}. Pattern avoidance is often easier to visualise on a plot where the plot of a Cayley permutation  $\pi_1 \pi_2 \cdots \pi_n$ is the set of points $(i, \pi_i)$. For example, the plot of the Cayley permutation $1 2 1 3 2 1 4$ is shown in Figure~\ref{fig:rgf} with an occurrence of $2213$ highlighted.
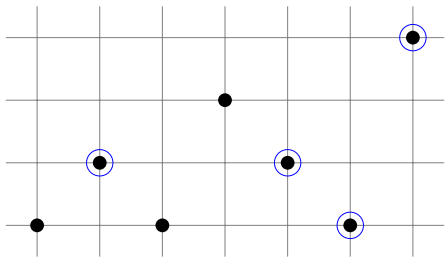
\begin{figure}[h]
  \centering
  \resizebox{6cm}{!}{
    \begin{tikzpicture}
      \draw[step=1cm,gray,thin] (0.5,0.5) grid (7.5,4.5);
      \foreach \x/\y in {1/1, 2/2, 3/1, 4/3, 5/2, 6/1, 7/4}
      \fill[black] (\x,\y) circle (0.11cm);
      \foreach \x/\y in {2/2, 5/2, 6/1, 7/4}
      \draw[blue] (\x,\y) circle (6pt);
    \end{tikzpicture}
  }
  \caption{The plot of the Cayley permutation $1 2 1 3 2 1 4$ with points from an occurrence of $2213$ circled.}
  \label{fig:rgf}
\end{figure}

The definitions of pattern avoidance and containment hold for RGFs as they are also Cayley permutations.
An RGF $\sigma$ avoids a set of Cayley permutations $\Pi$ if $\sigma$ avoids every Cayley permutation in $\Pi$, for example for the set $\Pi = \{121, 221\}$, the RGF $11233$ avoids $\Pi$ but $112133$ does not because it contains the pattern $121$ (even though it avoids $221$).

This notion of pattern avoidance on set partitions was first introduced by Klazar~\cite{Klazar1996,Klazar2000a,Klazar2000b}, where he refers to RGFs as \emph{canonical sequential form}. Other types of pattern avoidance include allowing reordering of the blocks~\cite{Sagan2006, Goyt2008, Goyt2009}; removing the partitions and employing pattern avoidance on the permutation that remains~\cite{Qiu2018}; and looking at pattern avoidance across blocks of a set partition where each part of a pattern must occur in a different block~\cite{Godbole2012,Chen2013,Kasraoui2013}. This last form of pattern avoidance corresponds to avoiding the inverse of the pattern in the equivalent Cayley permutation.

For a set of Cayley permutations $\Pi$, we define the \emph{RGF class} $\mathcal{R}(\Pi)$ to be the set of all RGFs avoiding the patterns in $\Pi$. In this context, we call $\Pi$ an \emph{avoiding set} of $\mathcal{R}(\Pi)$. An avoiding set of an RGF class is not unique, for example $\mathcal{R}(121) = \mathcal{R}(21)$. We describe an RGF class as \emph{finitely-based} if there exists a finite set of Cayley permutations $\Pi$ such that the class is $\mathcal{R}(\Pi)$. 

Sagan~\cite{Sagan2006} began the enumeration of RGFs as a method to enumerate unordered set partitions. Two RGF classes $\mathcal{R}(\Pi)$, $\mathcal{R}(\Pi')$ are \emph{Wilf-equivalent} if for all $n \in \mathbb{N}$, there are the same number of RGFs in $\mathcal{R}(\Pi)$ as in $\mathcal{R}(\Pi')$ of size $n$. Jelínek and Mansour~\cite{Jelínek2008} found all Wilf-equivalence classes of sets of RGFs avoiding one RGF of size up to seven and Mansour and Shattuck~\cite{Mansour2011a,Mansour2011b,Mansour2012} enumerated the RGF classes $\mathcal{R}(1222, 1212)$, $\mathcal{R}(1212, 12221)$ and $\mathcal{R}(1222, 12323)$. 
We enumerate RGF classes by adapting the vertical and horizontal insertion encodings of Cayley permutations~\cite{Bean2025}, which was first introduced by Albert, Linton and Ru\v{s}kuc to enumerate permutations~\cite{Albert2005}.
In this paper, we also give conditions for when the insertion encoding of an RGF class of either type forms a regular language.
To do this, we first introduce some key definitions and notation.

We define a size $n+m$ Cayley permutation $\pi$ to be a \emph{horizontal juxtaposition} of a Cayley permutation $\sigma$ of size $n$ and a Cayley permutation $\tau$ of size $m$ if the values at the smallest $n$ indices of $\pi$ standardise to $\sigma$ and the values at the largest $m$ indices of $\pi$ standardise to $\tau$. The Cayley permutation $1424 2534$ is a horizontal juxtaposition of the permutations $1323$ and $1423$ as $1424$ standardises to $1323$ and $2534$ standardises to $1423$. This is shown in Figure~\ref{fig: hor jux} with a dashed line between the two parts.

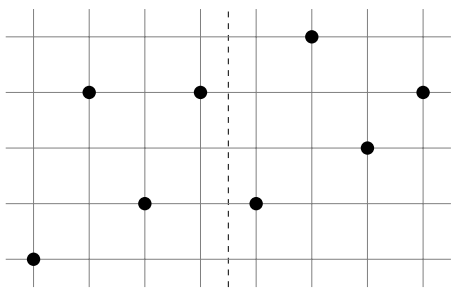
\begin{figure}[h]
  \centering
  \resizebox{6cm}{!}{%
    \begin{tikzpicture}
      \draw[dashed] (4.5, 0.5) -- (4.5, 5.5);
      \draw[step=1cm,gray,thin] (0.5,0.5) grid (8.5,5.5);
      \foreach \x/\y in {1/1, 2/4, 3/2, 4/4, 5/2, 6/5, 7/3, 8/4}
      \fill[black] (\x,\y) circle (0.12cm);
    \end{tikzpicture}
  }%
  \caption{$14242534$ is a horizontal juxtaposition of the Cayley permutations $1323$ and $1423$.}
  \label{fig: hor jux}
\end{figure}

There are four classes of Cayley permutations of particular interest which are the horizontal juxtapositions of increasing or decreasing sequences. These types will be denoted $\mathcal{H}_{A,B}$ where $A, B \in \{I, D \}$ and where $A$ denotes the type of sequence of the smallest values and $B$ denotes the type of sequence of the largest values. 

We also define a Cayley permutation $\pi$ with maximum value $n+m$ to be a \emph{vertical juxtaposition} of a Cayley permutation $\sigma$ with maximum value $n$ and a Cayley permutation $\tau$ with maximum value $m$ if the smallest $n$ values of $\pi$ standardise to $\sigma$ and the largest $m$ values of $\pi$ standardise to $\tau$. For example, the Cayley permutation $51521443$ is a vertical juxtaposition of the Cayley permutations $1213$ and $2211$ as $5544$ standardises to $221$. This is shown in Figure~\ref{fig: vert jux} with a dashed line between the largest and smallest values.

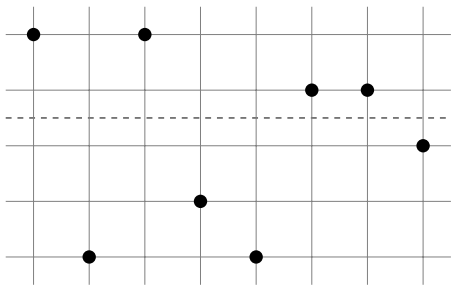
\begin{figure}[h]
  \centering
  \resizebox{6cm}{!}{%
    \begin{tikzpicture}
      \draw[dashed] (0.5, 3.5) -- (8.5, 3.5);
      \draw[step=1cm,gray,thin] (0.5,0.5) grid (8.5,5.5);
      \foreach \x/\y in {1/5, 2/1, 3/5, 4/2, 5/1, 6/4, 7/4, 8/3}
      \fill[black] (\x,\y) circle (0.12cm);
    \end{tikzpicture}
  }%
  \caption{$51521443$ is a vertical juxtaposition of the Cayley permutations $1213$ and $2211$.}
  \label{fig: vert jux}
\end{figure}

There are nine classes of Cayley permutations of particular interest which are the vertical juxtapositions of increasing, decreasing or constant sequences, denoted $I$, $D$ and $C$ respectively. These types will be denoted $\mathcal{V}_{A,B}$ where $A, B \in \{I, D, C \}$, $A$ denotes the type of sequence of the smallest values and $B$ denotes the type of sequence of the largest values. 

The \emph{horizontal concatenation} of two size $n$ Cayley permutations $\sigma$ and $\tau$ is the size $2n$ Cayley permutation of the form $\sigma \tau$.
There is exactly one horizontal concatenation of each even size in each of $\mathcal{H}_{I, I}$ and $\mathcal{H}_{I, D}$, for example the Cayley permutation $12344321$ is the horizontal concatenation of size 8 in $\mathcal{H}_{I, D}$.
A \emph{vertical alternation} of size $2n$ is the vertical juxtaposition of size $2n$ of the form $a_1b_1a_2b_2\cdots a_nb_n$ where $b_j > a_i$ for all $1 \leq j \leq n$ and $1 \leq i \leq n$. There is exactly one vertical alternation of each even size for each type of vertical juxtapositions, for example in $\mathcal{V}_{C, I}$ the size $6$ vertical alternation is $121314$.

These types of vertical and horizontal juxtapositions will be used to determine when the insertion encodings of an RGF class are regular.
In Section~\ref{sec:horizontal}, we adapt the horizontal insertion encoding to enumerate RGFs which can also be used to enumerate matchings and in Section~\ref{sec:vertical} we adapt the vertical insertion encoding.

\section{Horizontal insertion encoding}
\label{sec:horizontal}

A Cayley permutation can be generated by inserting values from left to right. We call this the \emph{(horizontal) evolution}, for example the evolution of the Cayley permutation 242143 is shown below where the final Cayley permutation can be read off from left to right in the final state.

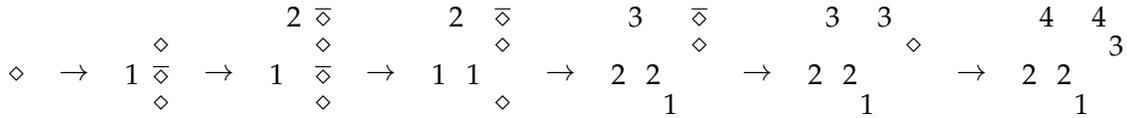
\begin{center}
  \resizebox{\linewidth}{!}{%
    \begin{tikzpicture}
        \node at (0,0) {$\diamond$};
        
        \node at (2em,0) {$\rightarrow$};
        
        \node at (4em, 0) {$1$};
        \node at (5em, 1em) {$\diamond$};
        \node at (5em, -1em) {$\diamond$};
        \node at (5em, 0em) {$\overline{\diamond}$};
        
        \node at (7em, 0) {$\rightarrow$};
        
        \node at (9em, 0) {$1$};
        \node at (9.6em, 2em) {$2$};
        \node at (10.6em, -1em) {$\diamond$};
        \node at (10.6em, 1em) {$\diamond$};
        \node at (10.6em, 2em) {$\overline{\diamond}$};
        \node at (10.6em, 0em) {$\overline{\diamond}$};
        
        \node at (12.6em, 0) {$\rightarrow$};
        
        \node at (14.6em, 0) {$1$};
        \node at (15.2em, 2em) {$2$};
        \node at (15.8em, 0em) {$1$};
        \node at (16.8em, -1em) {$\diamond$};
        \node at (16.8em, 1em) {$\diamond$};
        \node at (16.8em, 2em) {$\overline{\diamond}$};
        
        \node at (18.8em, 0) {$\rightarrow$};
        
        \node at (20.8em, 0) {$2$};
        \node at (21.4em, 2em) {$3$};
        \node at (22em, 0em) {$2$};
        \node at (22.6em, -1em) {$1$};
        \node at (23.6em, 1em) {$\diamond$};
        \node at (23.6em, 2em) {$\overline{\diamond}$};
        
        \node at (25.6em, 0) {$\rightarrow$};
        
        \node at (27.6em, 0) {$2$};
        \node at (28.2em, 2em) {$3$};
        \node at (28.8em, 0em) {$2$};
        \node at (29.4em, -1em) {$1$};
        \node at (30em, 2em) {$3$};
        \node at (31em, 1em) {$\diamond$};
        
        \node at (33em, 0) {$\rightarrow$};
        
        \node at (35em, 0) {$2$};
        \node at (35.6em, 2em) {$4$};
        \node at (36.2em, 0em) {$2$};
        \node at (36.8em, -1em) {$1$};
        \node at (37.4em, 2em) {$4$};
        \node at (38em, 1em) {$3$};
        
    \end{tikzpicture}
  }%
\end{center}

The intermediate states in an evolution are called \emph{configurations}; these are a mixture of the standardisation of the leftmost points in the Cayley permutation and some slots, represented by $\diamond$ or $\overline{\diamond}$. Each slot represents the promise of a value to be inserted. For a Cayley permutation $\pi$, the value inserted in $\diamond$ will be the leftmost occurrence of this value in $\pi$ whereas for $\overline{\diamond}$ the value inserted will be a repeat of the value at that height. We call $\overline{\diamond}$ a \emph{repeating} slot and $\diamond$ a \emph{new} slot. An evolution always begins with a single new slot and ends with a completed Cayley permutation.

There are four different ways of inserting into a slot denoted $u$, $m$, $d$ and $f$ describing where the point has been placed in relation to new slots that have been added. If the point is \emph{up} from a new slot then the insertion is described by $u$, if it is in the \emph{middle} of two new slots then it is $m$, if it is \emph{down} from a new slot then it is $d$ and if the previous slot was \emph{filled} with no new slots added then it is $f$.

Each insertion into a configuration can be described by the index of the slot which is inserted into, if there will be any more occurrences of the value inserted later in the Cayley permutation, and what the slot will be replaced with.
This is encoded by letters of the form $a_{i, j}$ where $i \in \mathbb{N}$ is the index of the slot, starting with the bottommost slot being slot 1. The index $j \in \{0, 1\}$ is $0$ if there will be no more occurrences of the value being inserted and is $1$ if there will be. The letter $a \in \{u, m, d, f\}$ describes the type of insertion, as defined in Table~\ref{tab:horizontal-insertions} which shows all possible insertions of the value $n$ into a slot at index $i$.
To ensure that the evolutions are unique, only $f$ can be inserted into a repeating slot. 

\begin{table}[h]
    \centering
    \begin{tabular}{|ccccc|ccccc|}
        \hline
        $u_{i, 0}$: & $\diamond$            & $\rightarrow$ & $n$ &            & $u_{i, 1}$: & $\diamond$            & $\rightarrow$ & $n$ & $\overline{\diamond}$ \\
                    &                       &               &     & $\diamond$ &             &                       &               &     & $\diamond$            \\
        \hline
                    &                       &               &     & $\diamond$ &             &                       &               &     & $\diamond$            \\
        $m_{i, 0}$: & $\diamond$            & $\rightarrow$ & $n$ &            & $m_{i, 1}:$ & $\diamond$            & $\rightarrow$ & $n$ & $\overline{\diamond}$ \\
                    &                       &               &     & $\diamond$ &             &                       &               &     & $\diamond$            \\
        \hline
                    &                       &               &     & $\diamond$ &             &                       &               &     & $\diamond$            \\
        $d_{i, 0}$: & $\diamond$            & $\rightarrow$ & $n$ &            & $d_{i, 1}$: & $\diamond$            & $\rightarrow$ & $n$ & $\overline{\diamond}$ \\
        \hline
        $f_{i, 0}$: & $\diamond$            & $\rightarrow$ & $n$ &            & $f_{i, 1}$: & $\diamond$            & $\rightarrow$ & $n$ & $\overline{\diamond}$ \\
        \hline
        $f_{i, 0}$: & $\overline{\diamond}$ & $\rightarrow$ & $n$ &            & $f_{i, 1}$: & $\overline{\diamond}$ & $\rightarrow$ & $n$ & $\overline{\diamond}$ \\
        \hline
    \end{tabular}
    \caption{Types of insertions of the value $n$ into a slot at index $i$. Only $f_{i, 0}$ and $f_{i, 1}$ can insert into $\overline{\diamond}$.}
    \label{tab:horizontal-insertions}
\end{table}

Using this notation, the horizontal insertion encoding for the Cayley permutation $242143$ is the word $m_{1,1}u_{3,1}f_{2,0}f_{1,0}f_{2,0}f_{1,0}$, the evolution for which is shown below.

\begin{center}
  \resizebox{\linewidth}{!}{%
    \begin{tikzpicture}
        \node at (0,0) {$\diamond$};
        
        \node at (2em,0) {$\rightarrow$};
        \node at (2em,1em) {$m_{1, 1}$};
        
        \node at (4em, 0) {$1$};
        \node at (5em, 1em) {$\diamond$};
        \node at (5em, -1em) {$\diamond$};
        \node at (5em, 0em) {$\overline{\diamond}$};
        
        \node at (7em, 0) {$\rightarrow$};
        \node at (7em,1em) {$u_{3, 1}$};
        
        \node at (9em, 0) {$1$};
        \node at (9.6em, 2em) {$2$};
        \node at (10.6em, -1em) {$\diamond$};
        \node at (10.6em, 1em) {$\diamond$};
        \node at (10.6em, 2em) {$\overline{\diamond}$};
        \node at (10.6em, 0em) {$\overline{\diamond}$};
        
        \node at (12.6em, 0) {$\rightarrow$};
        \node at (12.6em,1em) {$f_{2, 0}$};
        
        \node at (14.6em, 0) {$1$};
        \node at (15.2em, 2em) {$2$};
        \node at (15.8em, 0em) {$1$};
        \node at (16.8em, -1em) {$\diamond$};
        \node at (16.8em, 1em) {$\diamond$};
        \node at (16.8em, 2em) {$\overline{\diamond}$};
        
        \node at (18.8em, 0) {$\rightarrow$};
        \node at (18.8em,1em) {$f_{1,0}$};
        
        \node at (20.8em, 0) {$2$};
        \node at (21.4em, 2em) {$3$};
        \node at (22em, 0em) {$2$};
        \node at (22.6em, -1em) {$1$};
        \node at (23.6em, 1em) {$\diamond$};
        \node at (23.6em, 2em) {$\overline{\diamond}$};
        
        \node at (25.6em, 0) {$\rightarrow$};
        \node at (25.6em,1em) {$f_{2, 0}$};
        
        \node at (27.6em, 0) {$2$};
        \node at (28.2em, 2em) {$3$};
        \node at (28.8em, 0em) {$2$};
        \node at (29.4em, -1em) {$1$};
        \node at (30em, 2em) {$3$};
        \node at (31em, 1em) {$\diamond$};
        
        \node at (33em, 0) {$\rightarrow$};
        \node at (33em,1em) {$f_{1, 0}$};
        
        \node at (35em, 0) {$2$};
        \node at (35.6em, 2em) {$4$};
        \node at (36.2em, 0em) {$2$};
        \node at (36.8em, -1em) {$1$};
        \node at (37.4em, 2em) {$4$};
        \node at (38em, 1em) {$3$};
        
    \end{tikzpicture}
  }%
\end{center}

For a Cayley permutation $\pi$, the word corresponding to the evolution of $\pi$ is called the horizontal insertion encoding of $\pi$. The horizontal insertion encoding of a set of Cayley permutations $S$ is the language consisting of the horizontal insertion encodings of the Cayley permutations in $S$.

For the horizontal insertion encoding of RGFs, there can never be a new slot below a point as a new slot implies a value smaller than the point being placed has not occurred yet. So, the only letters possible are $d$ and $f$. The horizontal insertion encoding of RGFs is the subset of the horizontal insertion encoding of Cayley permutations that only uses the letters shown in Table~\ref{tab: canon horizontal insertion encoding letters}, proven in Proposition~\ref{prop: horizontal RGF letters}.

\begin{table}[h]
    \centering
    \begin{tabular}{|cccl|cccl|}
        \hline
                     &                                      &     & $\diamond$ &              &                                      &     & $\diamond$            \\
        $d_{i, 0}: $ & $\diamond \: \rightarrow$            & $n$ &            & $d_{i, 1}: $ & $\diamond \: \rightarrow$            & $n$ & $\overline{\diamond}$ \\
        \hline
        $f_{i, 0}: $ & $\diamond \: \rightarrow$            & $n$ &            & $f_{i, 1}: $ & $\diamond \: \rightarrow$            & $n$ & $\overline{\diamond}$ \\
        \hline
        $f_{i, 0}: $ & $\overline{\diamond} \: \rightarrow$ & $n$ &            & $f_{i, 1}: $ & $\overline{\diamond} \: \rightarrow$ & $n$ & $\overline{\diamond}$ \\
        \hline
    \end{tabular}
    \caption{Types of insertions of the value $n$ into a slot at index $i$ to create RGFs.}
    \label{tab: canon horizontal insertion encoding letters}
\end{table}

\begin{proposition}
    \label{prop: horizontal RGF letters}
    A Cayley permutation is an RGF if and only if its horizontal insertion encoding only uses the letters in Table~\ref{tab: canon horizontal insertion encoding letters}.
\end{proposition}
\begin{proof}
    First note that for a Cayley permutation $\pi$ to be an RGF it is both necessary and sufficient that for every value $n$ in $\pi$ and every value $k<n$, the leftmost occurrence of $k$ is at a smaller index than every occurrence of $n$ in $\pi$.
    
    Let $\pi$ be a Cayley permutation whose horizontal insertion encoding uses a letter not in Table~\ref{tab: canon horizontal insertion encoding letters}. Then there is a letter of the form $u_{i,j}$ or $m_{i,j}$ for some $i$ and $j$ in the horizontal insertion encoding of $\pi$. Suppose such a letter was inserted at index $k$ in $\pi$. For either letter, a new slot is created below the value inserted. At least one value is inserted into this new slot to create $\pi$. The first value inserted into this new slot will be a leftmost occurrence of a value in $\pi$ which is less than $\pi(k)$ but at an index larger than $k$ in $\pi$. Therefore, $\pi$ is not an RGF.
    
    Now, let $\pi$ be a Cayley permutation whose horizontal insertion encoding only uses the letters in Table~\ref{tab: canon horizontal insertion encoding letters}. Then every time a new value is inserted there are no slots below it, so there can be no occurrence of a value smaller than it which has a leftmost occurrence at a larger index. Therefore, $\pi$ is an RGF.
\end{proof}

The horizontal insertion encoding for the RGF $121331$ is given by the word $d_{1,1}d_{2,0}f_{1,1}f_{2,1}f_{2,0}f_{1,0}$, the evolution for which is shown below.

\begin{center}
  \resizebox{\linewidth}{!}{%
    \begin{tikzpicture}
        \node at (0,0) {$\diamond$};
        
        \node at (2em,0) {$\rightarrow$};
        \node at (2em,1em) {$d_{1, 1}$};
        
        \node at (4em, 0) {$1$};
        \node at (5em, 1em) {$\diamond$};
        \node at (5em, 0) {$\overline{\diamond}$};
        
        \node at (7em, 0) {$\rightarrow$};
        \node at (7em,1em) {$d_{2, 0}$};
        
        \node at (9em, 0) {$1$};
        \node at (9.6em, 1em) {$2$};
        \node at (10.6em, 2em) {$\diamond$};
        \node at (10.6em, 0) {$\overline{\diamond}$};
        
        \node at (12.6em, 0) {$\rightarrow$};
        \node at (12.6em,1em) {$f_{1,1}$};
        
        \node at (14.6em, 0) {$1$};
        \node at (15.2em, 1em) {$2$};
        \node at (15.8em, 0em) {$1$};
        \node at (16.8em, 2em) {$\diamond$};
        \node at (16.8em, 0) {$\overline{\diamond}$};
        
        \node at (18.8em, 0) {$\rightarrow$};
        \node at (18.8em,1em) {$f_{2, 1}$};
        
        \node at (20.8em, 0) {$1$};
        \node at (21.4em, 1em) {$2$};
        \node at (22em, 0em) {$1$};
        \node at (22.6em, 2em) {$3$};
        \node at (23.6em, 2em) {$\overline{\diamond}$};
        \node at (23.6em, 0) {$\overline{\diamond}$};
        
        \node at (25.6em, 0) {$\rightarrow$};
        \node at (25.6em,1em) {$f_{2, 0}$};
        
        \node at (27.6em, 0) {$1$};
        \node at (28.2em, 1em) {$2$};
        \node at (28.8em, 0em) {$1$};
        \node at (29.4em, 2em) {$3$};
        \node at (30em, 2em) {$3$};
        \node at (31em, 0) {$\overline{\diamond}$};
        
        \node at (33em, 0) {$\rightarrow$};
        \node at (33em,1em) {$f_{1, 0}$};
        
        \node at (35em, 0) {$1$};
        \node at (35.6em, 1em) {$2$};
        \node at (36.2em, 0em) {$1$};
        \node at (36.8em, 2em) {$3$};
        \node at (37.4em, 2em) {$3$};
        \node at (38em, 0em) {$1$};
    \end{tikzpicture}
  }%
\end{center}

As in~\cite[Section 4]{Bean2025}, for an RGF class to have a regular (horizontal) insertion encoding there must be a bound on the number of slots in any configuration in the evolution of any RGF in the class, a condition called being \emph{slot-bounded}.
An algorithm presented in~\cite[Section 5]{Bean2025} can be used to enumerate finitely-based Cayley permutation classes with a regular insertion encoding. As a subset of the letters are used for RGFs, this algorithm can be adapted to enumerate finitely-based RGF classes, an implementation for which can be found on GitHub~\cite{Bean_cperms_ins_enc_2025}. 
This shows that being slot-bounded is also a sufficient condition for an RGF class to have a regular (horizontal) insertion encoding. We will now give a characterisation of when an RGF class is slot-bounded.

Let $\mathcal{SB}_H(k)$ be the set of all RGFs with evolutions that contain configurations with at most $k$ slots.
For an RGF $\pi$, if a configuration in the evolution of $\pi$ has $k+1$ slots then the topmost slot could be a new slot or a repeating slot, but all other slots are repeating slots as otherwise there would be a new slot below a placed point, contradicting that $\pi$ is an RGF.

For a configuration $c$, we define the \emph{derivations} of $c$ as the set of all Cayley permutations which have $c$ in its evolution. 
As the only letter which increases the number of slots is $d_{i,1}$, for a configuration $c$ to have $k+1$ slots it must be described by the word consisting of at least $k$ many $d_{i,1}$. These insertions create an occurrence of a strictly increasing sequence of size $k$ in every derivation of $c$.
Each of the slots in $c$ must eventually be filled with at least one $f_{i, 0}$. As there are $k+1$ many slots, these $k+1$ many $f_{i, 0}$ form an occurrence of a permutation of size $k+1$ in any derivation of $c$.

The set of Cayley permutations with horizontal insertion encoding $d_{1,1}d_{2,1}\ldots d_{k,1}$ followed by $k+1$ many $f_{i, 0}$ are the minimal Cayley permutations which have a configuration with $k+1$ slots in their evolution.
These Cayley permutations are of size $2k+1$ and height $k+1$ where the first $k$ indices form a strictly increasing sequence of the values 1 to $k$ and the last $k+1$ indices are some permutation of the values 1 to $k+1$. Proposition \ref{prop: horizontal canonical SB(k) basis} proves that an avoiding set of $\mathcal{SB}_H(k)$ can be written as all Cayley permutations of this form.

\begin{proposition}
    \label{prop: horizontal canonical SB(k) basis}
    For each positive integer $k$, the set $\mathcal{SB}_H(k)$ is the set of RGFs that avoid all size $2k+1$ and height $k+1$ Cayley permutations where the first $k$ indices form a strictly increasing sequence of the values 1 to $k$ and the last $k+1$ indices are some permutation of the values 1 to $k+1$. 
\end{proposition}
\begin{proof}
    These Cayley permutations are RGFs and in their evolutions once the increasing sequence has been inserted in the first $k$ indices there remains $k+1$ slots, so these RGFs are not in $\mathcal{SB}_H(k)$.
    
    Let $\Pi$ be the set of size $2k+1$ and height $k+1$ Cayley permutations where the first $k$ indices form a strictly increasing sequence of the values 1 to $k$ and the last $k+1$ indices are some permutation of the values 1 to $k+1$. 
    For an RGF $\sigma \not\in \mathcal{SB}_H(k)$, there is a configuration in its evolution, say $c$, which has $k+1$ slots. 
    Apart from the topmost slot, these are all repeating slots so represent the promise of a repeat of a value which has already been placed in $c$.
    Taking the leftmost occurrence of these values creates an increasing sequence as the first occurrence of every value occurs in increasing order in an RGF.
    As there are at least $k$ many repeating slots, there are at least $k$ many values in this increasing sequence. Name this sequence $\pi_L$. 
    
    For every slot in $c$, take one point from $\sigma$ which is yet to be inserted in that slot. As there are $k+1$ slots, this creates a sequence of size $k+1$ with the smallest $k$ values being the same as those in $\pi_L$. Call this sequence $\pi_R$.
    Together, the occurrence of $\pi_L$ and $\pi_R$ in $\sigma$ creates an occurrence of a Cayley permutation of size $2k+1$ and height $k+1$ where the first $k$ indices form a strictly increasing sequence of the values 1 to $k$ and the last $k+1$ indices are some permutation of the values 1 to $k+1$, hence an element of $\Pi$.
    Therefore, any RGF which is not in $\mathcal{SB}_H(k)$ contains an element of $\Pi$, hence an avoiding set of $\mathcal{SB}_H(k)$ can be written as $\Pi$. 
\end{proof}

Proposition~\ref{prop: horizontal canonical SB(k) basis} allows us to prove a condition for when an RGF class is a subclass of $\mathcal{SB}_H(k)$ for some $k$, hence is regular, in Theorem \ref{thm: canonical classes regular condition}. The proof relies on Lemma~\ref{lem: horizontal concatenation contains jux} which gives a relation between horizontal concatenations and horizontal juxtapositions.

\begin{lemma}
    \label{lem: horizontal concatenation contains jux}
    A horizontal concatenation of type $\mathcal{H}_{I, I}$ or $\mathcal{H}_{I, D}$ and of size $2n$ contains every horizontal juxtaposition of the same type of size up to $n$.
\end{lemma}
\begin{proof}
    Let $a_1 a_2 \cdots a_n b_1 b_2 \cdots b_n$ be a horizontal concatenation of size $2n$ in $\mathcal{H}_{I, I}$ or $\mathcal{H}_{I, D}$. Take $\pi$ to be an arbitrary horizontal juxtaposition of $\sigma$ and $\tau$ of the same type where $|\sigma| = k$, $|\tau| = l$ and $k + l = m$ where $m \leq n$. 
    For each value $v$ in $\pi$, if $v$ is in the first $k$ indices of $\pi$ then select the value $a_i$ such that $a_i = v$. If $v$ is in the last $l$ indices then select the value $b_j$ such that $b_j = v$. There can be at most $n$ different values in $\pi$ and as there are also $n$ different values in the horizontal concatenation of size $2n$, this selection of values is a subset of the horizontal concatenation which standardises to $\pi$.
    Therefore, a horizontal concatenation of size $2n$ in $\mathcal{H}_{I, I}$ or $\mathcal{H}_{I, D}$ contains every horizontal juxtaposition of the same type of size up to $n$.
\end{proof}

\begin{theorem}
    \label{thm: canonical classes regular condition}
    An RGF class $\mathcal{R}$ is a subclass of $\mathcal{SB}_H(k)$ for some $k$ if and only if $\mathcal{R}$ avoids a Cayley permutation from each of $\mathcal{H}_{I,I}$ and $\mathcal{H}_{I,D}$.
\end{theorem}
\begin{proof}
    Suppose that $\mathcal{R} \subseteq \mathcal{SB}_H(k)$ for some $k$. Then, by Proposition~\ref{prop: horizontal canonical SB(k) basis}, $\mathcal{R}$ avoids Cayley permutations of size $2k+1$ and height $k+1$ where the first $k$ indices form a strictly increasing sequence of the values 1 to $k$ and the last $k+1$ indices are some permutation of the values 1 to $k+1$.
    In particular, it avoids Cayley permutations of this form where the last $k+1$ indices are increasing or decreasing which are in $\mathcal{H}_{I,I}$ and $\mathcal{H}_{I,D}$ respectively.
    Therefore, $\mathcal{R}$ avoids a Cayley permutation from $\mathcal{H}_{I,I}$ and from $\mathcal{H}_{I,D}$.
    
    Now, suppose $\mathcal{R}$ is a class of RGFs which avoid a Cayley permutation from $\mathcal{H}_{I,I}$ and from $\mathcal{H}_{I,D}$. 
    Suppose that the longest of these two Cayley permutations is of size $n$ for some fixed $n$. Taking the avoiding set $\Pi$ of $\mathcal{SB}_H(k)$ to be as defined in Proposition~\ref{prop: horizontal canonical SB(k) basis}, we will show that there exists some integer $k$ such that every element in $\Pi$ contains a Cayley permutation from $\mathcal{H}_{I,I}$ or $\mathcal{H}_{I,D}$ that is avoided by $\mathcal{R}$, so $\mathcal{R} \subseteq \mathcal{SB}_H(k)$ for this $k$.
    
    We can take $k$ to be $n^2$ and $\sigma$ be any element of $\Pi$.
    Split $\sigma$ into the values $\pi_L$ and the values $\pi_R$ where $\pi_L$ is the first $k$ indices and $\pi_R$ is the last $k+1$ indices as before. In $\pi_R$, take the smallest $k$ values and call this $\pi_R'$. This sequence has size $k = n^2$. By the Erd\H{o}s-Szekeres Theorem~\cite{Erdos1935}, there is a monotone subsequence of size $n$ which is a subset of these points, which we call $\tau$. 
    Note that all the values in $\pi_R'$ have had a repeated value in $\pi_L$.
    
    Next, take the same $n$ values from $\pi_L$ which are in $\tau$ and call them sequence $\omega$. These already form an increasing sequence. Together, the subsequence $\omega \tau$ forms a horizontal concatenation of type $\mathcal{H}_{I,I}$ or $\mathcal{H}_{I,D}$ and of size $2n$ which is a subset of $\sigma$. By Lemma~\ref{lem: horizontal concatenation contains jux}, this horizontal concatenation contains every horizontal juxtaposition of the same type up to size $n$, including one avoided by $\mathcal{R}$.
    As $\sigma$ was arbitrary, this is true for all elements of $\Pi$ and so $\mathcal{R} \subseteq \mathcal{SB}_H(k)$.
\end{proof}

In the final part of this section we will look at \emph{(perfect) matchings}, also known as chord diagrams \cite{Nabergall2022}. A matching is a graph on $2m$ vertices such that every vertex has degree 1. They are in bijection with unordered set partitions such that every block is of size 2, hence also in bijection with RGFs with exactly two occurrences of each value.

The letters for the horizontal insertion encoding of RGFs in bijection with matchings is a subset of the letters for RGFs generally. When inserting into a new slot, as there will be a repeat of that value later there is a repeating slot created, so insertions into $\diamond$ are of the form $a_{i, 1}$. As each repeat occurs exactly once, repeating slots are filled with exactly one value, hence insertions into $\overline{\diamond}$ are of the form $f_{i, 0}$. Table~\ref{tab: matching horizontal insertion encoding letters} shows all possible types of insertions of the value $n$ into a slot for matchings.

\begin{table}[h]
    \centering
    \begin{tabular}{|cccl|}
        \hline
                     &                                      &     & $\diamond$            \\
        $d_{i, 1}: $ & $\diamond \: \rightarrow$            & $n$ & $\overline{\diamond}$ \\ 
        \hline
        $f_{i, 1}: $ & $\diamond \: \rightarrow$            & $n$ & $\overline{\diamond}$ \\
        \hline
        $f_{i, 0}: $ & $\overline{\diamond} \: \rightarrow$ & $n$ &                       \\
        \hline
    \end{tabular}
    \caption{Types of insertions of the value $n$ into a slot at index $i$ to create RGFs of matchings.}
    \label{tab: matching horizontal insertion encoding letters}
\end{table}

Using this reduced alphabet, the horizontal insertion encoding for the RGF 122313 of a matching is given by the word $d_{1,1}d_{2,1}f_{2,0}f_{2,1}f_{1,0}f_{1,0}$ as shown below.

\begin{center}
  \resizebox{\linewidth}{!}{%
    \begin{tikzpicture}
        \node at (0,0) {$\diamond$};
        
        \node at (2em,0) {$\rightarrow$};
        \node at (2em,1em) {$d_{1, 1}$};
        
        \node at (4em, 0) {$1$};
        \node at (5em, 1em) {$\diamond$};
        \node at (5em, 0) {$\overline{\diamond}$};
        
        \node at (7em, 0) {$\rightarrow$};
        \node at (7em,1em) {$d_{2, 1}$};
        
        \node at (9em, 0) {$1$};
        \node at (9.6em, 1em) {$2$};
        \node at (10.6em, 2em) {$\diamond$};
        \node at (10.6em, 1em) {$\overline{\diamond}$};
        \node at (10.6em, 0) {$\overline{\diamond}$};
        
        \node at (12.6em, 0) {$\rightarrow$};
        \node at (12.6em,1em) {$f_{2,0}$};
        
        \node at (14.6em, 0) {$1$};
        \node at (15.2em, 1em) {$2$};
        \node at (15.8em, 1em) {$2$};
        \node at (16.8em, 2em) {$\diamond$};
        \node at (16.8em, 0) {$\overline{\diamond}$};
        
        \node at (18.8em, 0) {$\rightarrow$};
        \node at (18.8em,1em) {$f_{2, 1}$};
        
        \node at (20.8em, 0) {$1$};
        \node at (21.4em, 1em) {$2$};
        \node at (22em, 1em) {$2$};
        \node at (22.6em, 2em) {$3$};
        \node at (23.6em, 2em) {$\overline{\diamond}$};
        \node at (23.6em, 0) {$\overline{\diamond}$};
        
        \node at (25.6em, 0) {$\rightarrow$};
        \node at (25.6em,1em) {$f_{1, 0}$};
        
        \node at (27.6em, 0) {$1$};
        \node at (28.2em, 1em) {$2$};
        \node at (28.8em, 1em) {$2$};
        \node at (29.4em, 2em) {$3$};
        \node at (30em, 0em) {$1$};
        \node at (31em, 0) {$\overline{\diamond}$};
        
        \node at (33em, 0) {$\rightarrow$};
        \node at (33em,1em) {$f_{1, 0}$};
        
        \node at (35em, 0) {$1$};
        \node at (35.6em, 1em) {$2$};
        \node at (36.2em, 1em) {$2$};
        \node at (36.8em, 2em) {$3$};
        \node at (37.4em, 0em) {$1$};
        \node at (38em, 2em) {$3$};
    \end{tikzpicture}
  }%
\end{center}

Borrowing notation from Nabergall~\cite{Nabergall2022}, for a set of Cayley permutations $\Pi$ we denote the set of RGFs in bijection with matchings which avoid all Cayley permutations in $\Pi$ by $\mathcal{D}(\Pi)$.
We also define $\mathcal{SBM}_H(k)$ to be the set of all RGFs of matchings with evolutions that have at most $k$ slots in any configuration. 
An avoiding set of $\mathcal{SBM}_H(k)$ can be written the same as the avoiding set of $\mathcal{SB}_H(k)$ defined in Proposition~\ref{prop: horizontal canonical SB(k) basis} and the proof of Proposition~\ref{prop: horizontal canonical SB(k) basis} holds for matchings therefore Corollary~\ref{cor: slot bounded matchings} is a corollary to Proposition~\ref{prop: horizontal canonical SB(k) basis}.

\begin{corollary}
    \label{cor: slot bounded matchings}
    For each integer $k$, an avoiding set $\Pi$ of $\mathcal{SBM}_H(k)$ can be written as the set of all Cayley permutations of size $2k+1$ and height $k+1$ where the first $k$ indices form a strictly increasing sequence of the values 1 to $k$ and the last $k+1$ indices are some permutation of the values 1 to $k+1$.
\end{corollary}

Using Corollary~\ref{cor: slot bounded matchings}, we can prove a condition for when a class $\mathcal{D}(\Pi)$ is a subclass of $\mathcal{SBM}_H(k)$ for some $k$, hence is regular, in Corollary~\ref{cor: matching condition to be reg}. 
When this is the case, by using an adaption of the algorithm in~\cite[Section 5]{Bean2025} we can enumerate the classes $\mathcal{D}(\Pi)$ which have a regular horizontal insertion encoding.
Corollary~\ref{cor: matching condition to be reg} is a corollary to Theorem~\ref{thm: canonical classes regular condition} as the argument of the proof is the same.

\begin{corollary}
    \label{cor: matching condition to be reg}
    A class $\mathcal{D}(\Pi)$ is a subclass of $\mathcal{SBM}_H(k)$ for some $k$ if and only if $\mathcal{D}(\Pi)$ avoids a Cayley permutation from each of $\mathcal{H}_{I,I}$ and $\mathcal{H}_{I,D}$.
\end{corollary}

\section{Vertical insertion encoding}
\label{sec:vertical}

Cayley permutations can be created by inserting values from smallest to largest. To make this procedure unique, we insert leftmost occurrences of values first. We call this process the (vertical) evolution of a Cayley permutation. For example, the evolution of the Cayley permutation 242143 is shown below.

\begin{table}[h]
  \centering
  \begin{tabular}{cc}
     & $\diamond$                  \\
     & $\diamond~1~\diamond$       \\
     & $2~\diamond~1~\diamond$     \\
     & $2~\diamond~2~1~\diamond$  \\
     & $2~\diamond~2~1~\diamond~3$ \\
     & $2~4~2~1~\diamond~3$        \\
     & $2~4~2~1~4~3$               \\
  \end{tabular}
\end{table}

As before, the intermediate steps are called configurations.
Configurations contain the smallest, leftmost values in the Cayley permutation and only one type of slot. Evolutions still begin with a single empty slot and end with a completed Cayley permutation. 
For a configuration $c$, we define the \emph{derivations} of $c$ as the set of all Cayley permutations which have $c$ in its evolution, so 242143 is a derivation of the configuration $2~\diamond~2~1~\diamond~3$, as are 242153 and 2421443 for example.

There are four different types of insertions into a slot denoted by $\ell$, $m$, $r$, and $f$ where $\ell$ implies the new value is on the \emph{left} of a new slot, $r$ implies it is on the \emph{right} of a new slot, $m$ implies it is in the \emph{middle} of two new slots and $f$ implies that the slot has been \emph{filled} with the new value, so no more new slots have been placed. These are shown in Figure~\ref{tab:letters} where $n$ is the new value being inserted.

\begin{figure}[h]
  \centering
  \begin{tabular}{cccc}
    $\ell$: & $\diamond$ & $\rightarrow$ & $n \diamond$          \\
    $m$:    & $\diamond$ & $\rightarrow$ & $\diamond n \diamond$ \\
    $r$:    & $\diamond$ & $\rightarrow$ & $\diamond n$          \\
    $f$:    & $\diamond$ & $\rightarrow$ & $n$                   \\
  \end{tabular}
  \caption{The four ways to insert a value $n$ into a slot.}
  \label{tab:letters}
\end{figure}

Each insertion can be described by the index of the slot being inserted into, if the value inserted is the same as the last value that was inserted or one larger, and how the slot is replaced.  We encode these insertions by letters of the form $a_{i, j}$ where $a \in \{\ell, m, r, f\}$, $i \in \mathbb{N}$ is the index of the slot, starting with the leftmost slot being slot 1, and $j \in \{0, 1\}$ is $0$ if the value being inserted is the same as the last value which was inserted else 1 if it has increased. By convention, inserting into the initial slot in the first configuration of an evolution is an increase in value. The vertical insertion encoding for 242143 is the word $m_{1,1}\ell_{1,1}r_{1,0}r_{2,1}f_{1,1}f_{1,0}$, the evolution for which is shown again below with the encoding.

\begin{table}[h]
  \centering
  \begin{tabular}{ccc}
     & $\diamond$                  &              \\
     & $\diamond~1~\diamond$       & $m_{1,1}$    \\
     & $2~\diamond~1~\diamond$     & $\ell_{1,1}$ \\
     & $2~\diamond~2~1~\diamond$  & $r_{1,0}$    \\
     & $2~\diamond~2~1~\diamond~3$ & $r_{2,1}$    \\
     & $2~4~2~1~\diamond~3$        & $f_{1,1}$    \\
     & $2~4~2~1~4~3$               & $f_{1,0}$    \\
  \end{tabular}
\end{table}

For each Cayley permutation $\pi$, the vertical insertion encoding of $\pi$ is the unique word corresponding to the evolution of $\pi$. For a set of Cayley permutations $S$, the vertical insertion encoding of $S$ is the language consisting of the vertical insertion encodings of the Cayley permutations in $S$.

In the evolution of an RGF, any new maximum value inserted will be in the leftmost slot of a configuration otherwise in all the derivations of the new configuration there will be a larger value to the left of the new maximum value. Therefore, the vertical insertion encoding of RGFs will contain only the letters in Proposition~\ref{prop: canonical vertical insertion encoding}.

\begin{proposition}
  \label{prop: canonical vertical insertion encoding}
  A Cayley permutation is an RGF if and only if its vertical insertion encoding contains only the letters $f_{1,1}$, $\ell_{1,1}$, and $a_{i, 0}$ for $a \in \{f, \ell, m, r\}$ and $i \in \mathbb{N}$.
\end{proposition}
\begin{proof}
  Let $\pi$ be a Cayley permutation whose vertical insertion encoding contains a letter not of the form $f_{1,1}$, $\ell_{1,1}$, or $a_{i, 0}$ for $a \in \{f, \ell, m, r\}$ and $i \in \mathbb{N}$. Then the vertical insertion encoding of $\pi$ contains a letter $b$ of the form $r_{i, 1}$ or $m_{i,1}$ where $i \in \mathbb{N}$ or $f_{i,1}$, $\ell_{i,1}$ where $i \neq 1$. These are insertions of a new maximum with $r$ or $m$, or the insertion of a new maximum not in the first slot. Either way, in the evolution of $\pi$ the insertion for $b$ will be a new maximum $n$ with at least one slot at a smaller index. Each of these slots will eventually be filled to obtain $\pi$, for example with value $m$. As each of these slots are at a smaller index than $n$ and $m \neq n$ we have $m > n$. Therefore, in $\pi$ there exists an occurrence of $m$ with no occurrence of $n$ at a smaller index even though $n < m$, so $\pi$ is not an RGF.
  
  Now, let $\pi$ be a Cayley permutation whose vertical insertion encoding contains only the letters $f_{1,1}$, $\ell_{1,1}$, and $a_{i, 0}$ for $a \in \{f, \ell, m, r\}$ and $i \in \mathbb{N}$. Then in the evolution of $\pi$, every new maximum value is inserted with $f$ or $\ell$ and only into the first slot of a configuration. So, for every new maximum value $n$ inserted there are no slots to the left of it. There must have also been an occurrence of every value smaller than $n$ inserted into the first slot in the configuration and these must all be at smaller indices than the first occurrence of $n$. Therefore, for the first occurrence of any value $k$ in $\pi$, there must be an occurrence of every value smaller than $k$ at smaller indices, so $\pi$ is an RGF.
\end{proof}

Once again, the slot-bounded classes are precisely those which have a regular vertical insertion encoding. By restricting the alphabet, the algorithm in \cite[Section 5]{Bean2025} to compute the regular vertical insertion encoding of Cayley permutations can be adapted to compute the regular vertical insertion encoding of RGFs. For the remainder of this section, we will give a characterisation of the RGF classes which have a regular vertical insertion encoding. We define $\mathcal{SB}_V(k)$ to be the set of $k$ slot-bounded RGFs.

\begin{proposition}
  \label{prop:basis-canonical-slot-bounded}
  For each positive integer $k$, the set $\mathcal{SB}_V(k)$ is the set of RGFs that avoid all size $2k + n + 1$ Cayley permutations which are derivations of configurations of the form 
  \begin{equation*}
    1~2~\cdots~n~\diamond~a_1~\diamond~a_2~\diamond~\cdots~\diamond~a_k~\diamond
  \end{equation*} for any size $k$ Cayley permutation $a_1 a_2 \cdots a_k$ where $n$ is the maximum value in $a_1 a_2 \cdots a_k$.
\end{proposition}
\begin{proof}
  We will prove both directions by contrapositive. First, let $\pi$ be an RGF which contains a Cayley permutation $\sigma$ which is a derivation of a configuration of the form in the proposition.
  In the evolution of $\pi$ there is a configuration $c$ where the values corresponding to $n a_1 a_2 \cdots a_k$ in $\pi$ have just been placed. Due to the occurrence of $\sigma$ in $\pi$, there is an occurrence of at least one value after each of the $k+1$ values which are yet to be placed. Hence, $c$ has $k+1$ many slots, so $\pi \not\in \mathcal{SB}_V(k)$.
  
  Take an RGF $\pi$ not in $\mathcal{SB}_V(k)$. The evolution of $\pi$ contains a configuration with exactly $k + 1$ slots as each insertion increases the number of slots by at most $1$. Pick a value between each of the $k + 1$ slots, say $a_1, a_2, \ldots, a_k$. As $\pi$ is an RGF, if $n$ is the maximum value in $a_1 a_2 \cdots a_k$, then before the first slot each value $1, 2, \ldots, n$ also appears. Pick the first occurrence of each of these values. Finally, pick a value of $\pi$ in each slot. Together, these $2k + n + 1$ values form a derivation of the configuration $1~2~\cdots~n~\diamond~a_1~\diamond~a_2~\diamond~\cdots~\diamond~a_k~\diamond$.
\end{proof}
An RGF class $\mathcal{R}$ is slot bounded if and only if it is a subclass of $\mathcal{SB}_V(k)$ for some $k$. We will now outline a check on the Cayley permutations avoided by $\mathcal{R}$ for when this is the case. 
Three classes of interest are $\mathcal{V}_{C, I}$, $\mathcal{V}_{C, D}$ and $\mathcal{V}_{C,C}$ and their vertical alternations, as defined in Section~\ref{sec:intro}. For each class of vertical juxtapositions, Lemma~\ref{rem: vertical alternation} gives a relation between the vertical alternations and other vertical juxtapositions in the class.

\begin{lemma}[{\cite[Lemma 4.2]{Bean2025}}]
  \label{rem: vertical alternation}
  For each $A, B \in \{ I, D, C \}$, the vertical alternation of type $V_{A, B}$ of size $2n$ contains all vertical juxtapositions of type $V_{A, B}$ up to size $n$.
\end{lemma}

Vertical juxtapositions can be thought of as grid classes. Informally, the grid class of a matrix $M$ is the set of all Cayley permutations whose plots can be partitioned into cells where the points in each cell form an increasing, decreasing or constant sequence, as determined by $M$.
To define this more formally we first introduce some notation. For a Cayley permutation $\pi$ where $|\pi| = n$, the height of $\pi$ is $m$ and sets $A\subseteq [n]$, $B \subseteq [m]$, we write $\pi[A \times B]$ to denote the subset of points of $\pi$ with indices in $A$ and values in $B$. For $\pi = 135641742$ we have $\pi[[3, 7] \times [4]]=41$ for example, which is shown in Figure~\ref{fig: subset of points} with a dashed rectangle around the points in $\pi[[3, 7] \times [4]]$.

\begin{figure}[h]
  \centering
  \resizebox{6cm}{!}{
    \begin{tikzpicture}
      \draw[step=1cm,gray,thin] (0.5,0.5) grid (9.5,7.5);
      \foreach \x/\y in {1/3, 2/1, 3/5, 4/6, 5/4, 6/1, 7/7, 8/4, 9/2}
      \fill[black] (\x,\y) circle (0.15cm);
      \draw[dashed]  (2.5,0.5) rectangle (7.5,4.5);
    \end{tikzpicture}
  }
  \caption{The Cayley permutation $135641742$ with the points in $\pi[[3, 7] \times [4]]$ shown in a dashed rectangle.}
  \label{fig: subset of points}
\end{figure}
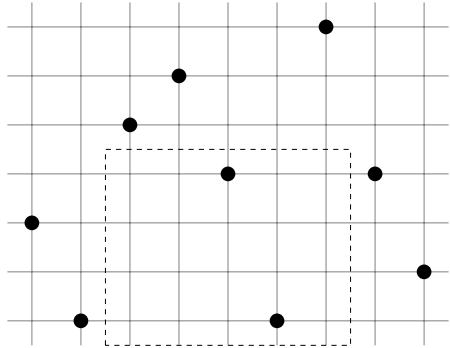

Borrowing notation from Huczynska and Vatter \cite{Huczynska2006}, we will index matrices beginning in the bottom left corner to match the way we plot a Cayley permutation, so the indices of a $2 \times 2$ matrix are
$
  \Big{(}
  \begin{smallmatrix}
    (1,2) & (2,2)\\
    (1,1) & (2,1)
  \end{smallmatrix}
  \Big{)}
$.
For a $t \times u$ matrix $M$, meaning $M$ has $t$ columns and $u$ rows, with entries in $\{I, D, C, 0\}$, we define an \emph{$M$-gridding} of a Cayley permutation $\pi$ where $\pi$ is of size $n$ with maximum value $m$ to be a sequence of columns $1 = c_1 \leq c_2 \leq \dots \leq c_{t+1} = n+1$ and a sequence of rows $1 = r_1 \leq r_2 \leq \dots \leq r_{u+1} = m+1$ such that for all $k \in [t]$ and $l \in [u]$, $\pi[[c_k, c_{k+1}) \times [r_l, r_{l+1})]$ is:
\begin{itemize}
  \item increasing if $M_{k,l} = I$,
  \item decreasing if $M_{k,l} = D$,
  \item constant if $M_{k,l} = C$,
  \item empty if $M_{k,l} = 0$.
\end{itemize}
The \emph{grid class} of a matrix $M$ is the set of all Cayley permutations which have an $M$-gridding. 
Vertical juxtapositions are grid classes where $M$ is a $1 \times 2$ matrix, in particular $V_{C, I}$ is the set of all Cayley permutations which have a gridding on $
  \big{(}
  \begin{smallmatrix}
    I\\
    C 
  \end{smallmatrix}
  \big{)}
$, $V_{C,D}$ is the set of all Cayley permutations which have a gridding on $
  \big{(}
  \begin{smallmatrix}
    D\\
    C 
  \end{smallmatrix}
  \big{)}
$, and $V_{C,C}$ is the set of all Cayley permutations which have a gridding on $
  \big{(}
  \begin{smallmatrix}
    C\\
    C 
  \end{smallmatrix}
  \big{)}
$.

For this paper, we will also define the grid classes
$\mathcal{G}_{A, B}$ where $M = \big{(} \begin{smallmatrix} 0 & B \\ I & A \end{smallmatrix} \big{)}$ for $A \in \{ I, D \}$ and $B \in \{ I, D, C \}$.
These six grid classes $\mathcal{G}_{I, I}$, $\mathcal{G}_{I, D}$, $\mathcal{G}_{I, C}$, $\mathcal{G}_{D, I}$, $\mathcal{G}_{D, D}$ and $\mathcal{G}_{D, C}$ as well as the grid classes $\mathcal{V}_{C, I}$, $\mathcal{V}_{C, D}$ and $\mathcal{V}_{C,C}$ are the nine classes that we will use to check when the vertical insertion encoding of an RGF class is regular. They are all shown in Figure~\ref{fig: grid classes}.

\begin{figure}[h]
  \centering
  \begin{subfigure}[b]{0.32\textwidth}
    \centering
    \resizebox{3cm}{!}{%
      \begin{tikzpicture} 
        \draw  (2.5,9.75) rectangle (7.5,4.75);
        \draw  (2.5,7.25) -- (7.5,7.25);
        \draw (5,9.75) -- (5,4.75);
        \draw  (3,5.25) -- (4.5,6.75);
        \draw  (5.5,7.75) -- (7,9.25);
        \draw (5.5,5.25) -- (7,6.75);
      \end{tikzpicture}
    }%
    \subcaption{Type $\mathcal{G}_{I,I}$.}
  \end{subfigure}
  \begin{subfigure}[b]{0.32\textwidth}
    \centering
    \resizebox{3cm}{!}{%
      \begin{tikzpicture}
        \draw  (2.5,9.75) rectangle (7.5,4.75);
        \draw  (2.5,7.25) -- (7.5,7.25);
        \draw  (5,9.75) -- (5,4.75);
        \draw  (3,5.25) -- (4.5,6.75);
        \draw (5.5,7.75) -- (7,9.25);
        \draw (5.5,6.75) -- (7,5.25);
      \end{tikzpicture}
    }%
    \subcaption{Type $\mathcal{G}_{D,I}$.}
  \end{subfigure}
  \begin{subfigure}[b]{0.16\textwidth}
    \centering
    \resizebox{1.5cm}{!}{%
      \begin{tikzpicture}
        \draw  (5,9.75) rectangle (7.5,4.75);
        \draw  (5,7.25) -- (7.5,7.25);
        \draw (5.5,7.75) -- (7,9.25);
        \draw (5.5,6) -- (7,6);
      \end{tikzpicture}
    }%
    \subcaption{Type  $\mathcal{V}_{C,I}$.}
    \label{fig: grid classes V CI}
  \end{subfigure}
  \begin{subfigure}[b]{0.32\textwidth}
    \centering
    \resizebox{3cm}{!}{%
      \begin{tikzpicture}
        \draw  (2.5,9.75) rectangle (7.5,4.75);
        \draw  (2.5,7.25) -- (7.5,7.25);
        \draw  (5,9.75) -- (5,4.75);
        \draw (5.5,8.5) -- (7,8.5);
        \draw  (3,5.25) -- (4.5,6.75);
        \draw (5.5,5.25) -- (7,6.75);
      \end{tikzpicture}
    }%
    \subcaption{Type $\mathcal{G}_{I,C}$.}
  \end{subfigure}
  \begin{subfigure}[b]{0.32\textwidth}
    \centering
    \resizebox{3cm}{!}{%
      \begin{tikzpicture}
        \draw  (2.5,9.75) rectangle (7.5,4.75);
        \draw (2.5,7.25) -- (7.5,7.25);
        \draw  (5,9.75) -- (5,4.75);
        \draw  (5.5,8.5) -- (7,8.5);
        \draw  (5.5,6.75) -- (7,5.25);
        \draw (3,5.25) -- (4.5,6.75);
      \end{tikzpicture}
    }%
    \subcaption{Type $\mathcal{G}_{D,C}$.}
  \end{subfigure}
  \begin{subfigure}[b]{0.16\textwidth}
    \centering
    \resizebox{1.5cm}{!}{%
      \begin{tikzpicture}
        \draw  (5,9.75) rectangle (7.5,4.75);
        \draw  (5,7.25) -- (7.5,7.25);
        \draw  (5.5,6) -- (7,6);
        \draw  (5.5,8.5) -- (7,8.5);
      \end{tikzpicture}
    }%
    \subcaption{Type $\mathcal{V}_{C,C}$.}
    \label{fig: grid classes V CC}
  \end{subfigure}
  \begin{subfigure}[b]{0.32\textwidth}
    \centering
    \resizebox{3cm}{!}{%
      \begin{tikzpicture} 
        \draw  (2.5,9.75) rectangle (7.5,4.75);
        \draw  (2.5,7.25) -- (7.5,7.25);
        \draw (5,9.75) -- (5,4.75);
        \draw  (3,5.25) -- (4.5,6.75);
        \draw (5.5,9.25) -- (7,7.75);
        \draw (5.5,5.25) -- (7,6.75);
      \end{tikzpicture}
    }%
    \subcaption{Type $\mathcal{G}_{I,D}$.}
  \end{subfigure}
  \begin{subfigure}[b]{0.32\textwidth}
    \centering
    \resizebox{3cm}{!}{%
      \begin{tikzpicture}
        \draw  (2.5,9.75) rectangle (7.5,4.75);
        \draw  (2.5,7.25) -- (7.5,7.25);
        \draw  (5,9.75) -- (5,4.75);
        \draw  (3,5.25) -- (4.5,6.75);
        \draw (5.5,9.25) -- (7,7.75);
        \draw (5.5,6.75) -- (7,5.25);
      \end{tikzpicture}
    }%
    \subcaption{Type $\mathcal{G}_{D,D}$.}
  \end{subfigure}
  \begin{subfigure}[b]{0.16\textwidth}
    \centering
    \resizebox{1.5cm}{!}{%
      \begin{tikzpicture}
        \draw  (5,9.75) rectangle (7.5,4.75);
        \draw  (5,7.25) -- (7.5,7.25);
        \draw  (5.5,6) -- (7,6);
        \draw (5.5,9.25) -- (7,7.75);
      \end{tikzpicture}
    }%
    \subcaption{Type $\mathcal{V}_{C,D}$.}
    \label{fig: grid classes V CD}
  \end{subfigure}
  \caption{Nine grid classes where $\diagup$ denotes increasing, $\diagdown$ denotes decreasing and $-$ denotes constant in that cell.}
  \label{fig: grid classes}
\end{figure}
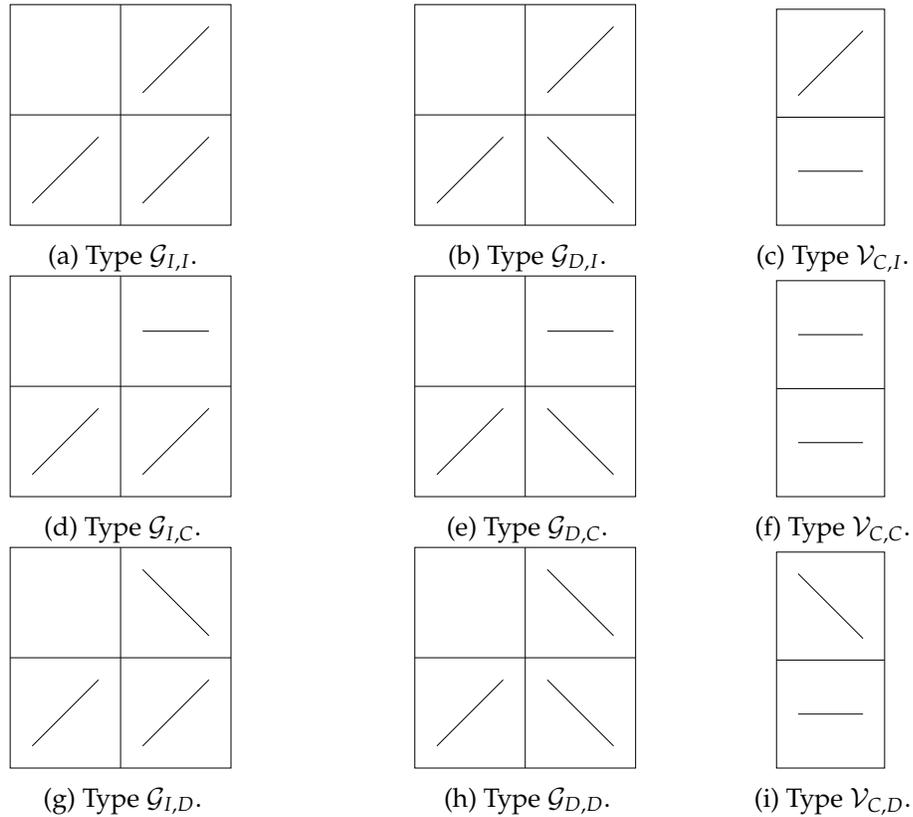

For $2 \times 2$ grid classes, we will call the set of points in each cell of a gridding a quadrant where $M_{1,1}$ is the $\alpha$ quadrant, $M_{2,1}$ is the $\beta$ quadrant, $M_{1,2}$ is the $\gamma$ quadrant and $M_{2,2}$ is the $\delta$ quadrant, hence  they are represented as $\big{(} \begin{smallmatrix} \gamma & \delta \\ \alpha & \beta \end{smallmatrix} \big{)}$.
For each grid class $\mathcal{G} \in \{ \mathcal{G}_{I,I}, \mathcal{G}_{I,C}, \mathcal{G}_{I,D}, \mathcal{G}_{D,I}, \mathcal{G}_{D,C}, \mathcal{G}_{D,D}  \}$, we define the \emph{$\mathcal{G}$-alternation} of size $3n$ to be the Cayley permutation in $\mathcal{G}$ of the form $123 \cdots n a_1 b_1 a_2 b_2 \cdots a_n b_n$ where $a_1 a_2 \cdots a_n$ has values $1$ to $n$ and the values in $b_1 b_2 \cdots b_n$ are strictly larger than $n$.
This has a $\mathcal{G}$-gridding where $a_1 a_2 \cdots a_n$ is in the $\beta$ quadrant, $b_1 b_2 \cdots b_n$ is in the $\delta$ quadrant and $12 3 \cdots n$ is in the $\alpha$ quadrant.
For example, the $\mathcal{G}_{I, I}$-alternation of size 9 is $123142536$, shown in Figure~\ref{fig: GII alternation} where the quadrants are represented by dashed lines and the points in the bottom left cell are in the $\alpha$ quadrant shown in red, points in the bottom right cell are in the $\beta$ quadrant shown in black and points in the top right cell are in the $\delta$ quadrant in blue.

\begin{figure}[h]
  \centering
  \resizebox{6cm}{!}{
    \begin{tikzpicture}
      \draw[step=1cm,gray,thin] (0.5,0.5) grid (9.5,6.5);
      \foreach \x/\y in {1/1, 2/2, 3/3}
      \fill[red] (\x,\y) circle (0.15cm);
      \foreach \x/\y in {4/1, 6/2,  8/3}
      \fill[black] (\x,\y) circle (0.15cm);
      \foreach \x/\y in {5/4, 7/5, 9/6}
      \fill[blue] (\x,\y) circle (0.15cm);
      \draw[dashed] (0.5, 3.5) -- (9.5, 3.5);
      \draw[dashed] (3.5, 0.5) -- (3.5, 6.5);
    \end{tikzpicture}
  }
  \caption{The $\mathcal{G}_{I,I}$-alternation of size $9$.}
  \label{fig: GII alternation}
\end{figure}
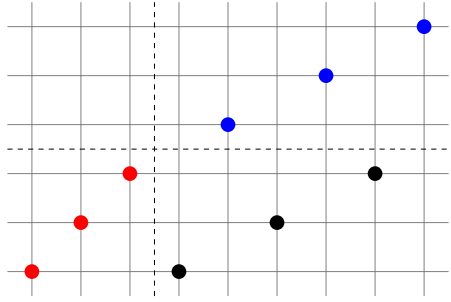

\begin{lemma}
  \label{lem: canon vertical alternation of size n, increasing}
  For each $\mathcal{G} \in  \{ \mathcal{G}_{I,I}, \mathcal{G}_{I,C}, \mathcal{G}_{I,D}, \mathcal{G}_{D,I}, \mathcal{G}_{D,C}, \mathcal{G}_{D,D} \}$, the $\mathcal{G}$-alternation of size $3n$ contains every Cayley permutation in $\mathcal{G}$ of size up to $n$.
\end{lemma}
\begin{proof}
  Let $\pi$ be a size $n$ Cayley permutation with maximum value $m$ in $\mathcal{G}$ for some fixed $\mathcal{G} \in  \{  \mathcal{G}_{I,I}, \mathcal{G}_{I,C}, \mathcal{G}_{I,D}, \mathcal{G}_{D,I}, \mathcal{G}_{D,C}, \mathcal{G}_{D,D} \}$. We will show that the $\mathcal{G}$-alternation $\sigma$ of size $3n$ in $\mathcal{G}$ contains $\pi$ by showing that it is sufficient to add at most $2n$ values (not necessarily distinct) to $\pi$ to create a $\mathcal{G}$-alternation.
  
  As $\pi \in \mathcal{G}$, there exists a $\mathcal{G}$-gridding of $\pi$. In this $\mathcal{G}$-gridding, suppose there are $k$ many points in the $\alpha$ quadrant, $j$ many points in the $\beta$ quadrant and $l$ many points in the $\delta$ quadrant where $k + j + l = n$. If there are values in the $\alpha$ quadrant which are not in the $\beta$ quadrant, then add these to the $\beta$ quadrant so that the values in the $\beta$ quadrant remain in increasing or decreasing order to create a new Cayley permutation $\pi' \in \mathcal{G}$. This adds $a$ many values where $0 \leq a \leq k$ as there are $k$ points in the $\alpha$ quadrant, so $|\pi'| = n + a$.
  
  In the $\beta$ and $\delta$ quadrants of $\pi'$ there are $j + l + a$ many points which form a vertical juxtaposition. By Lemma~\ref{rem: vertical alternation}, these points are contained in a vertical alternation of size $2(j + l + a)$, so by adding at most $j + l + a$ many values to $\pi'$, we can create a Cayley permutation $\tau \in \mathcal{G}$ of size $n+2a+j+l$ such that the points in the $\beta$ quadrant and the points in the $\delta$ quadrant perfectly interleave.
  
  In $\tau$, there are $j + l + a$ many points in the $\beta$ quadrant of which at least $a$ many are also in the $\alpha$ quadrant. Add any values in the $\beta$ quadrant which are not in the $\alpha$ quadrant to the $\alpha$ quadrant so that they form an increasing sequence, creating the Cayley permutation $\tau' \in \mathcal{G}$ which is a $\mathcal{G}$-alternation. There are at most $j+l$ many points in the $\beta$ quadrant which are not in the $\alpha$ quadrant, so we have added at most $j + l$ many values to create $\tau'$. Therefore, the size of $\tau'$ is at most $n+2a+2j+2l$. As $a \leq k$ and $l + j + k = n$, $\tau'$ is a $\mathcal{G}$-alternation of size $3n$ or smaller which contains $\pi$.

  As $\pi$ was arbitrary, every Cayley permutation in $\mathcal{G}$ of size $n$ is contained in a $\mathcal{G}$-alternation of size $3n$ or smaller, so the $\mathcal{G}$-alternation of size $3n$ contains every Cayley permutation in $\mathcal{G}$ of size up to $n$.
\end{proof}

This allows us to prove Theorem~\ref{them: vert slot bounded}, the main theorem of this section, which relies on Lemma~\ref{lem: sequence n^4 implies strict sequence n}, a result which we have proved in a previous paper~\cite[Lemma 4.2]{Bean2025} but now give a tighter bound.

\begin{lemma}
  \label{lem: sequence n^4 implies strict sequence n}
  Any sequence of size $n^3$ contains a subsequence of size $n$ that is either strictly increasing, strictly decreasing or constant.
\end{lemma}
\begin{proof}
  Let $\pi$ be a sequence of size $n^3$. Either $\pi$ contains $n$ occurrences of the same value, in which case we are done, or there is no value in $\pi$ which occurs $n$ times. In this case, there must be at least $n^2$ distinct values in $\pi$. Taking $\sigma$ to be the first occurrences of each of these values in $\pi$, by the Erd\H{o}s-Szekeres Theorem~\cite{Erdos1935} $\sigma$ contains either a strictly increasing or strictly decreasing subsequence of size $n$.
\end{proof}

We now fully classify when RGF classes are slot-bounded and therefore have a regular vertical insertion encoding in Theorem~\ref{them: vert slot bounded}.
The nine grid classes described in Theorem~\ref{them: vert slot bounded} are all represented in Figure~\ref{fig: grid classes} so an RGF class has a regular vertical insertion encoding if and only if it avoids Cayley permutations griddable on each of these nine grid classes.

\begin{theorem}
  \label{them: vert slot bounded}
  The class $\mathcal{R}$ is a subclass of $\mathcal{SB}_V(k)$ for some $k$ if and only if $\mathcal{R}$ avoids Cayley permutations which are griddable on each of the grid classes $\mathcal{G}_{I,I}$, $\mathcal{G}_{I,C}$, $\mathcal{G}_{I,D}$, $\mathcal{G}_{D,I}$, $\mathcal{G}_{D,C}$, $\mathcal{G}_{D,D}$, $\mathcal{V}_{C,I}$, $\mathcal{V}_{C,C}$ and $\mathcal{V}_{C,D}$.
\end{theorem}
\begin{proof}
  Suppose that $\mathcal{R} \subseteq \mathcal{SB}_V(k)$, then it avoids derivations of configurations of the form 
  \begin{equation*}
    1~2~\cdots~n~\diamond~a_1~\diamond~a_2~\diamond~\cdots~\diamond~a_k~\diamond
  \end{equation*}
  from Proposition~\ref{prop:basis-canonical-slot-bounded}. Take $a_1a_2\cdots a_k$ to be the strictly increasing or strictly decreasing sequence of size $k$ and let $n$ be the maximum value in this sequence. Insert values into the slots so that they form a strictly increasing, strictly decreasing or constant sequence. These derivations are $\mathcal{G}$-alternations from each of the six grid classes $\mathcal{G}_{I,I}$, $\mathcal{G}_{I,C}$, $\mathcal{G}_{I,D}$, $\mathcal{G}_{D,I}$, $\mathcal{G}_{D,C}$ and $\mathcal{G}_{D,D}$. 
  Similarly, take $a_1a_2\cdots a_k$ to be $1^k$ and $n$ to be 1. Insert values into the slots so that they form a strictly increasing, strictly decreasing or constant sequence. These derivations are vertical alternations from the classes$\mathcal{V}_{C,I}$, $\mathcal{V}_{C,C}$ and $\mathcal{V}_{C,D}$. As each of these nine sequences are not in $\mathcal{R}$, $\mathcal{R}$ must avoid Cayley permutations griddable on each of these nine grid classes.
  
  Now, suppose $\mathcal{R}$ is a class which avoids Cayley permutations griddable on each of the nine classes. Suppose that the largest Cayley permutation griddable on one of the nine classes which is avoided by $\mathcal{R}$ is of size $n$ for some fixed $n$. We will show that there exists some integer $k$ such that every Cayley permutation in the avoiding set $\Pi$ of $\mathcal{SB}_V(k)$ defined in Proposition~\ref{prop:basis-canonical-slot-bounded} contains a Cayley permutation that is avoided by $\mathcal{R}$ and which is griddable on one of the nine grid classes, so $\mathcal{R} \subseteq \mathcal{SB}_V(k)$ for this $k$.
  
  Take $k$ to be $n^{9}$ and $\sigma$ be any element of $\Pi$, writing $\sigma$ as $$12...c_mb_1a_1b_2a_2b_3...b_ka_kb_{k+1}$$ where $c_m$ is the maximum value of all the $a_i$'s. 
  There exists a $2 \times 2$ gridding of $\sigma$ where the $\alpha$ quadrant contains the sequence $12 \ldots c_m$, the $\beta$ quadrant contains the sequence $a_1 a_2 \ldots a_k$, the $\gamma$ quadrant is empty and the $\delta$ quadrant contains the sequence $b_1 b_2 \ldots b_{k+1}$.
  In this gridding, there are $k +1= n^{9}+1$ many points in the $\delta$ quadrant, so by Lemma~\ref{lem: sequence n^4 implies strict sequence n} there exists a subsequence of size $n^3$ which is strictly increasing, strictly decreasing or constant. We will denote this sequence as $b_{i_1} b_{i_2} \ldots b_{i_{n^3}}$.
  
  Next, we take a subsequence of points from the $\beta$ quadrant of size $n^3$ that perfectly alternate with the subset of points from the $\delta$ quadrant so that they form the pattern $a_{i_1}b_{i_1}a_{i_2}b_{i_2}\ldots a_{i_{n^3}}b_{i_{n^3}}$. Among these $n^3$ points from the $\beta$ quadrant there exists a subsequence of size $n$ which is strictly increasing, strictly decreasing or constant by Lemma~\ref{lem: sequence n^4 implies strict sequence n}. We will denote this sequence $a_{j_1} a_{j_2} \ldots a_{j_n}$.
  Taking these and only $n$ of the points from the $\delta$ quadrant which perfectly interleave with them creates a pattern of size $2n$ of the form $a_{j_1}b_{j_1}a_{j_2}b_{j_2}\ldots a_{j_n}b_{j_n}$ across the $\beta$ and $\delta$ quadrants.
  
  Finally, let $m'$ be the maximum of the sequence $a_{j_1} a_{j_2} \ldots a_{j_n}$. As $c_m$ is the maximum value of all the points in the $\beta$ quadrant and $m' \leq c_m$, there exists a subset of the points in the $\alpha$ quadrant of size $m'$ which contain the same values as the sequence $a_{j_1} a_{j_2} \ldots a_{j_n}$ in the $\beta$ quadrant. We will denote this sequence $c_1 c_2 \ldots c_{m'}$.
  
  This creates the pattern $\pi = c_1 c_2 \ldots c_{m'}a_{j_1}b_{j_1}a_{j_2}b_{j_2}\ldots a_{j_n}b_{j_n}$ where the sequence $c_1 c_2 \ldots c_{m'}$ is in the $\alpha$ quadrant, $a_{j_1} a_{j_2} \ldots a_{j_n}$ is in the $\beta$ quadrant and $b_{j_1} b_{j_2} \ldots b_{j_n}$ is in the $\delta$ quadrant. If the sequence in the $\beta$ quadrant is constant then $\pi$ is a vertical alternation of size $2n+1$ in $\mathcal{V}_{C,I}$, $\mathcal{V}_{C,C}$ or $\mathcal{V}_{C,D}$, so by Lemma~\ref{rem: vertical alternation} it contains every vertical juxtaposition of the same type up to size $n$, hence every Cayley permutation griddable on $\mathcal{V}_{C,I}$, $\mathcal{V}_{C,C}$ or $\mathcal{V}_{C,D}$ of size up to $n$. Otherwise, $\pi$ is a $\mathcal{G}$-alternation of size $3n$ in $\mathcal{G}_{I,I}$, $\mathcal{G}_{I,C}$, $\mathcal{G}_{I,D}$, $\mathcal{G}_{D,I}$, $\mathcal{G}_{D,C}$ or $\mathcal{G}_{D,D}$ and by Lemma~\ref{lem: canon vertical alternation of size n, increasing} it contains every Cayley permutation griddable on $\mathcal{G}_{I,I}$, $\mathcal{G}_{I,C}$, $\mathcal{G}_{I,D}$, $\mathcal{G}_{D,I}$, $\mathcal{G}_{D,C}$ or $\mathcal{G}_{D,D}$ of size up to $n$. Either way, $\pi$ must contain a Cayley permutation that is avoided by $\mathcal{R}$ and is griddable on one of the nine grid classes. As this element $\sigma$ of $\Pi$ contains a Cayley permutation which is avoided by $\mathcal{R}$ and as $\sigma$ was arbitrary, every element of $\Pi$ contains a Cayley permutation which is avoided by $\mathcal{R}$. Therefore, $\mathcal{R} \subseteq \mathcal{SB}_V(k)$.
\end{proof}

\section{Concluding remarks}
\label{sec:conclusion}

In this paper, we have adapted the two methods of enumerating Cayley permutations via the horizontal and vertical insertion encoding to RGFs and RGFs of matchings. For both methods we have classified when the insertion encoding forms a regular language. 
An algorithm for generating the rational generating functions for the regular insertion encodings can be found on GitHub~\cite{Bean_cperms_ins_enc_2025}. We have also given a method of enumerating RGFs corresponding to matchings using the horizontal insertion encoding and classified when this insertion encoding forms a regular language.

Table~\ref{tab:results for size 3} shows the number of RGF classes which can be defined by avoiding a set of size 3 Cayley permutations which have a regular insertion encoding. Every class with 5 or more basis elements has a regular insertion encoding of some type. Of the 8191 RGF classes defined by avoiding a set of size 3 patterns, 8161 have a regular insertion encoding. By finding their generating functions, we found that these form 72 different Wilf-equivalent classes with largest asymptotic growth rate of $\frac{5+\sqrt{5}}{2}$. 


\begin{table}[h]
    \centering
    \begin{tabular}{|c|c|c|c|c|}
        \hline
        Size of basis & Number of classes & Regular vertical  & Regular horizontal  & Either \\
        &  & insertion encoding &insertion encoding &  \\
        \hline
        1	&	13	&	2	&	5	&	6	\\
        2	&	78	&	33	&	58	&	65	\\
        3	&	286	&	221	&	262	&	278	\\
        4	&	715	&	668	&	699	&	713	\\
        5	&	1287	&	1269	&	1281	&	1287	\\      
        \hline
    \end{tabular}
    \caption{The number of RGF classes defined by avoiding a set of size 3 patterns. In each column, the number of classes with a regular vertical insertion encoding and a regular horizontal insertion encoding are shown. The last column shows the number of classes with either a regular vertical or horizontal insertion encoding.}
    \label{tab:results for size 3}
\end{table}

Because we have used Cayley permutations to define the bases, some of these classes are the same sets of RGFs. By defining our bases as RGFs instead, these 8191 classes correspond to only 206 different sets of RGFs. Of these, 190 have a regular insertion encoding. For example, $\Av(1212,1213)$ is the representation of both $\Av(212,213,312)$ and $\Av(212, 213)$ as RGFs. Interestingly, for this case only $\Av(212,213,312)$ passes our check for a regular vertical insertion encoding whereas $\Av(212, 213)$ does not.

This leaves 16 classes which do not have a regular insertion encoding. Jelínek and Mansour \cite{Jelínek2008} enumerated 5 of these classes and Mansour and Shattuck enumerated another two in \cite{Mansour2011a} and \cite{Mansour2011c}. 
The RGF class avoiding 111 is in bijection with partial matchings, in particular, $\Av(111,1212)$ is the number of non-crossing partial matchings which by \cite[Remark 1.2]{Guo2025} are enumerated by the Motzkin numbers and $\Av(111,1221)$ is the number of non-nesting partial matchings, also enumerated by the Motzkin numbers.
Additionally, $\Av(12312)$ and $\Av(12321)$ represent 2-distant non-crossing partitions and 2-distant non-nesting partitions respectively, which have been enumerated by Drake and Kim \cite[Theorem 5.3]{Drake2009}.
This leaves only 5 classes which have not been enumerated. They are shown in Table~\ref{tab:remaining classes} with a representative class defined by avoiding size 3 patterns in the first column, their corresponding basis as defined by the avoidance of RGFs in the second column, and the first few terms of their enumeration in the final column. Each of these classes represent 2-distant non-crossing or non-nesting partitions with extra pattern avoidance.
\begin{table}[h]
    \centering
    \begin{tabular}{|c|c|c|}
    \hline
    Size 3 representative & Defined by RGFs & First 10 terms \\
     \hline
    $\Av(111,312)$ & $\Av(111,12312)$ & 1, 1, 2, 4, 10, 25, 68, 187, 534, 1544, 4554\\
     \hline
    $\Av(111,321)$ & $\Av(111,12321)$ & 1, 1, 2, 4, 10, 25, 68, 187, 534, 1544, 4554\\
     \hline
    $\Av(211,312)$ & $\Av(1211,12312)$ & 1, 1, 2, 5, 14, 42, 132, 429, 1430, 4862, 16796\\
     \hline
    $\Av(211,321)$ & $\Av(1211,12321)$ & 1, 1, 2, 5, 14, 42, 132, 429, 1430, 4862, 16796\\
     \hline
    $\Av(213,312)$ & $\Av(1213,12312)$ & 1, 1, 2, 5, 14, 40, 113, 314, 860, 2329, 6254\\
    \hline
    \end{tabular}
    \caption{Sets of RGFs defined by avoiding size 3 Cayley permutations with unknown enumeration.}
    \label{tab:remaining classes}
\end{table}



In the original paper on the insertion encoding of permutations by Albert, Linton, and Ru\v{s}kuc~\cite{Albert2005}, the authors also classify when the insertion encoding forms a context-free language. This was not discussed in this paper as it does not lead to a method of automatically computing the algebraic generating functions but could be a direction for further research for both RGFs and RGFs of matchings.

Françon and Viennot~\cite{Françon1979} use the same core ideas as the insertion encoding to enumerate permutation classes and find statistics on them.
By using similar ideas as in this paper, the insertion encoding could be used to find statistics on RGFs and RGFs of matchings. Please see Campbell, Dahlberg, Dorward, Gerhard, Grubb, Pureell and Sagan~\cite{Campbell2016}, for example, for more information on RGF statistics.

There are many other techniques from permutation patterns which could be adapted to RGFs. For example, ideas from simple permutations~\cite{Albert2005b,Bassino2015}, the rationality of grid classes~\cite{Albert2013}, or enumerating RGF classes which avoid non-classical patterns~\cite{Baxter2011,Pudwell2010,Brändén2011,Ulfarsson2010}.

\bibliographystyle{apalike}
\bibliography{bibliography.bib}

@article{Erdos1935,
author = {Erd\H{o}s, Paul and Szekeres, George},
year = {1935},
pages = {463-470},
title = {A combinatorial problem in geometry},
volume = {2},
journal = {Compositio Mathematica}
}

@article{Sagan2006,
author = {Sagan, Bruce},
year = {2006},
month = {05},
pages = {},
title = {Pattern Avoidance in Set Partitions},
volume = {94},
journal = {Ars Combinatoria}
}

@article{Campbell2016,
author = {Campbell, Lindsey and Dahlberg, Samantha and Dorward, Robert and Gerhard, Jonathan and Grubb, Thomas and Purcell, Carlin and Sagan, Bruce},
year = {2016},
pages = {},
title = {Restricted growth function patterns and statistics},
volume = {100},
journal = {Advances in Applied Mathematics},
doi = {10.1016/j.aam.2018.05.002}
}

@article{Jelínek2008,
  author  = {Jelínek, Vít and Mansour, Toufik},
  year    = {2008},
  pages   = {39},
  title   = {On Pattern-Avoiding Partitions},
  volume  = {15},
  journal = {Electronic Journal of Combinatorics},
  doi     = {10.37236/763}
}

@article{Huczynska2006,
author = {Huczynska, Sophie and Vatter, Vincent},
year = {2006},
month = {03},
pages = {},
title = {Grid Classes and the {F}ibonacci Dichotomy for Restricted Permutations},
volume = {13},
journal = {Electronic Journal of Combinatorics},
doi = {10.37236/1080}
}

@article{Albert2005,
  author         = {Albert, MH and Linton, S and Ru\v{s}kuc, N},
  title          = {The insertion encoding of permutations},
  journal        = {Electronic Journal of Combinatorics},
  year           = {2005},
  volume         = {12},
  number         = {1},
  article-number = {R47},
  issn           = {1077-8926},
doi = {10.37236/1944},
  unique-id      = {WOS:000231975400003}
}

@article{Françon1979,
author = {Françon, Jean and Viennot, Gérard},
year = {1979},
month = {12},
pages = {21-35},
title = {Permutations selon leurs pics, creux, doubles montées et double descentes, nombres d'euler et nombres de Genocchi},
volume = {28},
journal = {Discrete Mathematics - DM},
doi = {10.1016/0012-365X(79)90182-1}
}

@misc{Bean_cperms_ins_enc_2025,
  author       = {Bean, Christian and Ollson, Abigail},
  year         = {2025},
  title        = {{cperms\_ins\_enc}},
  howpublished = {\url{https://github.com/Ollson2921/cperms\_ins\_enc}}
}

@misc{Bean2025,
author = {Bean, Christian and Bell, Paul C. and Ollson, Abigail},
year = {2025},
month = {05},
pages = {},
title = {The insertion encoding of {C}ayley permutations},
doi = {10.48550/arXiv.2505.08480},
howpublished = {\url{https://arxiv.org/abs/2505.08480v1}}
}

@article{Godbole2012,
  author  = {Godbole, Anant and Goyt, Adam and Herdan, Jennifer and Pudwell, Lara},
  year    = {2012},
  pages   = {},
  title   = {Pattern Avoidance in Ordered Set Partitions},
  volume  = {18},
  journal = {Annals of Combinatorics},
  doi     = {10.1007/s00026-014-0232-y}
}

@article{Chen2013,
  author  = {Chen, William and Dai, Alvin and Zhou, Dapao},
  year    = {2013},
  title   = {Ordered Partitions Avoiding a Permutation of Length 3},
  volume  = {36},
  journal = {European Journal of Combinatorics},
  doi     = {10.1016/j.ejc.2013.09.001}
}

@article{Kasraoui2013,
  author  = {Kasraoui, Anisse},
  year    = {2013},
  pages   = {},
  title   = {Pattern avoidance in ordered set partitions and words},
  volume  = {61},
  journal = {Advances in Applied Mathematics},
  doi     = {10.1016/j.aam.2014.08.004}
}

@article{Mansour2011a,
  author  = {Mansour, Toufik and Shattuck, Mark},
  year    = {2011},
  pages   = {1121},
  title   = {Pattern avoiding partitions and {M}otzkin left factors},
  volume  = {9},
  journal = {Central European Journal of Mathematics},
  doi     = {10.2478/s11533-011-0057-4}
}

@article{Mansour2011b,
  author  = {Mansour, Toufik and Shattuck, Mark},
  year    = {2011},
  pages   = {397},
  title   = {Pattern avoiding partitions, sequence {A}054391, and the kernel method},
  volume  = {6},
  journal = {Applications and Applied Mathematics}
}

@article{Mansour2011c,
author = {Mansour, Toufik and Shattuck, Mark},
year = {2011},
month = {01},
pages = {239-},
title = {Restricted partitions and generalized Catalan numbers},
volume = {22},
journal = {Pure Mathematics and Applications}
}

@article{Mansour2012,
  author  = {Mansour, Toufik and Shattuck, Mark},
  year    = {2012},
  pages   = {34},
  title   = {Pattern-Avoiding Set Partitions and {C}atalan Numbers},
  volume  = {18},
  journal = {Electronic Journal of Combinatorics},
  doi     = {10.37236/2048}
}

@article{Klazar1996,
  author  = {Klazar, Martin},
  year    = {1996},
  pages   = {53-68},
  title   = {On abab-Free and abba-Free Set Partitions},
  volume  = {17},
  journal = {European Journal of Combinatorics},
  doi     = {10.1006/eujc.1996.0005}
}

@article{Klazar2000a,
author = {Klazar, Martin},
year = {2000},
month = {02},
pages = {},
title = {Counting Pattern-free Set Partitions I: A Generalization of {S}tirling Numbers of the Second Kind},
volume = {21},
journal = {European Journal of Combinatorics},
doi = {10.1006/eujc.1999.0353}
}

@article{Klazar2000b,
author = {Klazar, Martin},
year = {2000},
month = {05},
pages = {},
title = {Counting Pattern-free Set Partitions II: Noncrossing and Other Hypergraphs},
volume = {7},
journal = {Electronic Journal of Combinatorics},
doi = {10.37236/1512}
}

@unknown{Guo2025,
author = {Guo, Peter},
year = {2025},
month = {11},
pages = {},
title = {Richardson tableaux and noncrossing partial matchings},
doi = {10.48550/arXiv.2511.15094}
}

@inproceedings{Drake2009,
  author = {Drake, Dan and Kim, Jang Soo},
  title = {{$k$}-distant crossings and nestings of matchings and partitions},
  booktitle = {21st International Conference on Formal Power Series and Algebraic Combinatorics (FPSAC 2009)},
  series = {DMTCS Proceedings},
  volume = {AK},
  year = {2009},
  doi = {10.46298/dmtcs.2746}
}

@article{Goyt2008,
  author  = {Goyt, Adam},
  year    = {2008},
  pages   = {95-114},
  title   = {Avoidance of Partitions of a Three-element Set},
  volume  = {41},
  journal = {Advances in Applied Mathematics},
  doi     = {10.1016/j.aam.2006.07.006}
}

@article{Goyt2009,
  author  = {Goyt, Adam and Sagan, Bruce},
  year    = {2009},
  pages   = {230-245},
  title   = {Set partition statistics and {F}ibonacci numbers},
  volume  = {30},
  journal = {European Journal of Combinatorics},
  doi     = {10.1016/j.ejc.2008.01.015}
}

@article{Qiu2018,
  author  = {Qiu, Dun and Remmel, Jeffery},
  year    = {2018},
  pages   = {},
  title   = {Patterns in words of ordered set partitions},
  volume  = {10},
  journal = {Journal of Combinatorics},
  doi     = {10.4310/JOC.2019.v10.n3.a2}
}

@phdthesis{Nabergall2022,
  author = {Lukas Nabergall},
  school = {University of Waterloo},
  title = {Enumerative perspectives on chord diagrams},
  year = {2022}
}

@article{Albert2005b,
author = {Albert, Michael and Atkinson, M.},
year = {2005},
month = {09},
pages = {1-15},
title = {Simple permutations and pattern restricted permutations},
volume = {300},
journal = {Discrete Mathematics},
doi = {10.1016/j.disc.2005.06.016}
}

@article{Albert2013,
author = {Albert, Michael and Atkinson, M. and Bouvel, Mathilde and Ruškuc, Nik and Vatter, Vincent},
year = {2011},
month = {08},
pages = {},
title = {Geometric grid classes of permutations},
volume = {365},
journal = {Transactions of the American Mathematical Society},
doi = {10.1090/S0002-9947-2013-05804-7}
}

@article{Baxter2011,
author = {Baxter, Andrew and Pudwell, Lara},
year = {2011},
month = {08},
pages = {},
title = {Enumeration schemes for vincular patterns},
volume = {312},
journal = {Discrete Mathematics},
doi = {10.1016/j.disc.2012.01.021}
}

@article{Pudwell2010,
author = {Pudwell, Lara},
year = {2010},
month = {02},
pages = {},
title = {Enumeration Schemes for Permutations Avoiding Barred Patterns},
volume = {17},
journal = {Electronic Journal of Combinatorics},
doi = {10.37236/301}
}

@article{Brändén2011,
author = {Brändén, Petter and Claesson, Anders},
year = {2011},
month = {02},
pages = {},
title = {Mesh Patterns and the Expansion of Permutation Statistics as Sums of Permutation Patterns},
volume = {18},
journal = {Electronic Journal of Combinatorics},
doi = {10.37236/2001}
}

@article{Ulfarsson2010,
author = {Ulfarsson, Henning},
year = {2010},
month = {02},
pages = {},
title = {A unification of permutation patterns related to {S}chubert varieties},
volume = {22},
journal = {FPSAC'10 - 22nd International Conference on Formal Power Series and Algebraic Combinatorics},
doi = {10.46298/dmtcs.2832}
}

@article{Bassino2015,
author = {Bassino, Frédérique and Bouvel, Mathilde and Pierrot, Adeline and Pivoteau, Carine and Rossin, Dominique},
year = {2015},
month = {06},
pages = {},
title = {An algorithm computing combinatorial specifications of permutation classes},
volume = {224},
journal = {Discrete Applied Mathematics},
doi = {10.1016/j.dam.2017.02.013}
}

\end{document}